\documentclass[a4paper, 12pt]{amsart}

\usepackage{geometry}             

\usepackage{float}
\usepackage{bm}
\usepackage{fullpage,xcolor}
\usepackage[mathscr]{euscript}
\usepackage[all]{xy}
\usepackage{epsfig}
\usepackage{amsfonts}
\usepackage{mathptmx}

\usepackage{graphicx}
\usepackage{amssymb}
\usepackage{amsmath}
\usepackage{amsthm}
\usepackage{mathrsfs}
\usepackage{epstopdf}
\usepackage{url}
\usepackage[msc-links,alphabetic]{amsrefs}
\usepackage{tikz}

\usepackage{color}
\makeatletter

\@addtoreset{equation}{section}
\makeatother 

\theoremstyle{plain}
\newtheorem{theorem}{Theorem}[section]

\newtheorem{proposition}[theorem]{Proposition}
\newtheorem{definition}[theorem]{Definition}
\newtheorem{lemma}[theorem]{Lemma}

\newtheorem{remark}[theorem]{Remark}
\newtheorem{example}[theorem]{Example}

\usepackage{latexsym}

\title[Classification of doubly periodic untwisted \\ $(p,q)$-weaves by their crossing number]
  {Classification of doubly periodic untwisted \\ $(p,q)$-weaves by their crossing number}

\author[Mizuki Fukuda, Motoko Kotani, Sonia Mahmoudi]{Mizuki Fukuda, Motoko Kotani, Sonia Mahmoudi}
\address{Advanced Institute for Materials Research, Tohoku University, 2-1-1 Katahira, Aoba-ku, Sendai, 980-8577, Japan}
\email{sonia.mahmoudi.mines@gmail.com}
\subjclass[2020]{57K10, 57K12 , 57M15 , 05A05} 
\keywords{weave, weaving diagram, crossing number, link in a thickened torus}
\date{February 3, 2022}
\thanks{This work is supported by a Research Fellowship from JST CREST Grant Number JPMJCR17J4.}

\begin{document}

\begin{abstract} 
A weave is the lift to the Euclidean thickened plane of a set of infinitely many planar crossed geodesics, 
that can be characterized by a number of sets of threads describing the organization of the non-intersecting curves, 
together with a set of crossing sequences representing the entanglements.
In this paper, the classification of a specific class of doubly periodic weaves, called untwisted $(p,q)$-weaves, is done by their crossing number, 
which is the minimum number of crossings that can possibly be found in a unit cell of its infinite weaving diagrams.
Such a diagram can be considered as a particular type of quadrivalent periodic planar graph with an over or under information at each vertex, 
whose unit cell corresponds to a link diagram in a thickened torus.
Moreover, considering that a weave is not uniquely defined by its sets of threads and its crossing sequences, 
we also specify the notion of equivalence classes by introducing a new parameter, called crossing matrix.
\end{abstract}

\maketitle

\section{Introduction}\label{sec:1}

A \textit{weave} is an embedding of infinite intertwined \textit{threads} in a thickened Euclidean plane
which differs from general links mainly because it does not contain any closed components \cite{ref1}. 
Many well-known weaves are doubly periodic objects and it is therefore natural to approach these structures using knot theory, 
as done by S. Grishanov, H. Morton et al. in \cite{ref4, ref5, ref6, ref16}, as well as A. Kawauchi in \cite{ref13}. 
Such periodic weaves can indeed be fully described by a repeating unit cell, that we call a \textit{weaving motif}, 
which can simply be seen as a particular type of link diagram in a thickened torus, due to the factorization of its group of translation symmetries. 
They attempted to describe and classify doubly periodic structures with this strategy by extending for example the Kauffman bracket \cite{ref4} or the multi-variable Alexander polynomials \cite{ref16}, 
considering that two such structures are equivalent if they satisfy the Reidemeister Theorem in the torus \cite{ref4,ref13}. 
However such topological invariants do not necessarily distinguish a weave from other types of doubly periodic entangled structures whose component are also sets homeomorphic to $\mathbb{R}$, 
and which are often found and compared in materials science for their different physical properties. 
Besides, some classical knot invariants relevant for the classification of these structures cannot be computed, such as the \textit{crossing number} \cite{ref1}.
This invariant is defined as the minimum number of \textit{crossings} that can be possibly found in a weaving motif and is useful to characterize the complexity of a weave,
but the unit cells of a given doubly periodic structure can be chosen in many different ways, possibly with a different number of crossings. 

In this paper, we introduce a formal topological definition of a class of weaves, generalizing existing physical structures in materials science, such as textiles \cite{ref18}.
First, we consider a set of geodesics embedded in the Euclidean plane $\mathbb{E}^2$ belonging to $N \geq 2$ disjoint color groups. 
These lines are such that two geodesics in the same color group do not intersect each other, namely they are parallel, or are in different color groups if they intersect once (Fig.~\ref{fig1}(A)).
Then, each intersection is given an over or under information, called \textit{crossing} as in knot theory (Fig.~\ref{fig1}(B)), 
according to a set of \textit{crossing sequences} which consist in a sequence of integers with minimal length describing the number of consecutive over or undercrossings for each geodesic. 
Finally, we define an \textit{untwisted weave} as the lift of these crossed geodesics to the Euclidean thickened plane $\mathbb{X}^3 = \mathbb{E}^2 \times  \mathrm{I}$, where $\mathrm{I} = [-1,1]$, 
with respect to the crossing information and such that the lifted geodesics, called \textit{threads}, do not intersect in $\mathbb{X}^3$ (Fig.~\ref{fig1}(C)). 
We will see in our future work that we can define any weave as an untwisted weave, possibly with twists, where a twist is created by cutting and gluing two consecutive parallel threads to introduce new crossings. 

\begin{figure}[h]
\centering
   \includegraphics[scale=1]{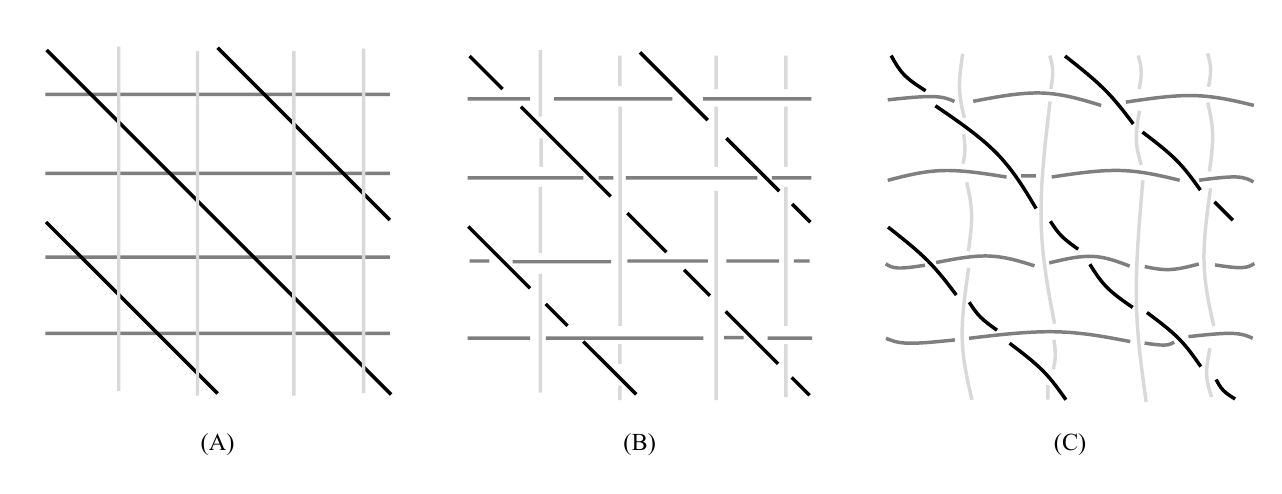}
      \caption{\label{fig1} 
(A) Sets of geodesics belonging to different color groups in the Euclidean plane. 
(B) Crossing information to each intersection in the Euclidean plane. 
(C) Untwisted weave in the thickened Euclidean plane.}
\end{figure}

Our approach makes it possible to distinguish a weave from other periodic complex entangled structures made of one-dimensional objects, 
such as a \textit{knit} \cite{ref16}, or a \textit{braid} \cite{ref14}, which consists of a single set of threads. 
We do not consider the case of entangled connected nets either, described by S.T. Hyde et al. in \cite{ref19}, however their cases of weaving of filament components match our weaves here.
Next, we will study a particular class of doubly periodic untwisted weaves, called \textit{$(p,q)$-weaves}, and classify them according to their crossing number using combinatorial arguments, 
which make the originality of our results. An extension to the case of weaves admitting twists is under study.
Note that we will not consider the theory of virtual knots \cite{ref12}, as done by Kawauchi \cite{ref13}, for our classification here.   

\begin{definition}\label{def:1-1} 
Let $i$, $j$ and $k$ be strictly positive integers.
A $(p,q)$-weave $\mathrm{W}$ is defined such that all its crossing sequences, possibly distinct, can be described by two positive integers $p_k$ and $q_k$. 
This means that if $C_{i,j} = (+p_k,-q_k)$ is the crossing sequence associated to the disjoint sets of threads $T_i$ and $T_j$ of $\mathrm{W}$, 
each thread of $T_i$ is cyclically $p_k$ consecutive times over the threads of $T_j$, followed by $q_k$ consecutive times under. 
Moreover, the complementary crossing sequence of $C_{i,j}$ is given by $C_{j,i} = (+q_k,-p_k)$. 
\end{definition}

By definition of a doubly periodic untwisted $(p,q)$-weave $\mathrm{W}$ as a lift of planar geodesics, 
any of its weaving motif is a set of \textit{essential} simple closed curves in a torus 
that can be characterized by a pair of coprime integers of type $(a,b)$, representing the slopes of the straight lines in $\mathbb{E}^2$. 
Recall that a closed curve is called essential if it is not homotopic to a point, a puncture, or a boundary component \cite{ref3}. 
Therefore, such a weave with $N \geq 2$ \textit{sets of threads}, defined by the respective color groups, can be (re)construct 
from a lift to the universal cover and translational symmetries of a set $\Sigma'$ of such curves on a thickened torus, 
by choosing $N$ couples of distinct coprime integers, together with a set $\Sigma$ of \textit{crossing sequences} to assign an over or under information to each crossing.
Then, by combining the crossing sequence $C_{i,j}$ of each pair of curves $(a_i,b_i)$ and $(a_j,b_j)$, representing the sets of threads $T_i$ and $T_j$, 
with their associated \textit{geometric intersection number} $|v_{i,j}| = |a_i \, b_j - a_j \, b_i |$, 
we can deduce the \textit{$(i,j)$-pairwise crossing number} of each pair $(T_i,T_j)$, 
defined as the minimum number of crossings necessary to encode the periodicity of the infinite structure, 
both on simple closed curves representing $T_i$ and $T_j$. 
Finally, we will prove that the \textit{total crossing number} of $\mathrm{W}$ is given by our following main Theorem, proved in Section 3.2. 

\begin{theorem}\label{thm:1-2} \textbf{(Total Crossing Number)}
Let $i, j \in \{1, \cdots , N\}$ distinct and $\mathcal{C}_{i,j}$ be the $(i,j)$-pairwise crossing numbers of a weaving diagram 
$D_{W_0}$ with $N$ sets of threads, characterized by the pair $(\Sigma',\Sigma)$. 
Let $(\mathcal{S}_{min})$, be the system of geometric intersection number equations, defined for integers $a_i$ and $b_i$, 
either coprime or such that one of them equal zero,  
satisfying for each equation that we can multiply both parts by a same even number $k_l$ 
if the two sets of threads implied on the equation cross ($k_l = 1$ otherwise), 
with $l \in \{1, \cdots , \frac{N (N - 1)}{2}\}$ being the index of the equation in the system.\\

 $ (\mathcal{S}_{min}) \left \{
   \begin{array}{r c l}
      \ k_1 \times |a_1 \ b_2 - a_2 \ b_1|   & = & k_1 \times \mathcal{C}_{1,2} \\
      . \\
      . \\
      . \\
      k_l \times |a_i \ b_j - a_j \ b_i|   & = & k_l \times \mathcal{C}_{i,j} \\
      . \\
      . \\
      .
   \end{array}
   \right .$
   
 ~ \\

Then, the total crossing number of $\mathrm{W}$ is given by,
$$ \mathcal{C} = \sum_{i<j = 1}^{N} \mathcal{C'}_{i,j} , $$ 
such that $\mathcal{C'}_{i,j} = k_l \ \mathcal{C}_{i,j}$, where $k_{l}$ are the smallest multipliers that give a solution to $(\mathcal{S}_{min})$.\\
\end{theorem}

However the same pair $(\Sigma',\Sigma)$ can generate two weaving diagrams that do not belong to the same equivalence class. 
This motivates the introduction of a new parameter $\Pi$ which encodes the organization of the crossings in a weaving motif, called a \textit{crossing matrix},   
and we will prove our second main Theorem in Section 3.1.

\begin{theorem}\label{thm:1-3} \textbf{(Equivalence Classes of doubly-periodic untwisted $(p,q)$-weaves)} 
Let $\mathrm{W}_1$ and $\mathrm{W}_2$ be two doubly periodic untwisted $(p,q)$-weaves with $N \geq 2$ sets of threads,
such that their corresponding regular projections are equivalent, up to isotopy of $\mathbb{E}^2$, and with the same set of crossing sequences.
Let $D_{W_1}$ and $D_{W_2}$ be two weaving motifs of same area of $\mathrm{W}_1$ and $\mathrm{W}_2$, respectively.
Then, $D_{W_1}$ and $D_{W_2}$ are equivalent if and only if their crossing matrices are pairwise equivalent.
\end{theorem}

Therefore, it is possible to construct non-equivalent doubly periodic untwisted $(p,q)$-weaves from a given one using a non-equivalent set of crossing matrices. 
However, to find the crossing number of such a weave, one must find the next smallest multipliers $k_{l}$ compatible 
with a solution of $(\mathcal{S}_{min})$ in Theorem 1.2 and corresponding to a weaving diagram with a set of crossing matrices equivalent to $\Pi$.\\


This paper begins by introducing the main definitions of untwisted weaves and weaving diagrams in Section 2.
Then, Section 3 discusses on the classification of doubly periodic untwisted $(p,q)$-weaves, considering 
the characterization of the equivalence classes in Section 3.1 and the computation of their crossing number in Section 3.2.
Finally, we apply our new results to the construction and classification of simple examples of square and kagome doubly periodic untwisted $(p,q)$-weaving diagrams by hand in Section 3.3. \\ 


~

\section{Untwisted weaves and their diagrams}\label{sec:2}
 
In this section, we will define an untwisted weave (Fig.~\ref{fig1}(C)) from a set of geodesics embedded in $\mathbb{E}^2$ (Fig.~\ref{fig1}(A)), 
with a crossing information to each intersection (Fig.~\ref{fig1}(B)), 
lifted to the topological ambient space $\mathbb{X}^3 = \mathbb{E}^2 \times  \mathrm{I}$, with $\mathrm{I} = [-1,1]$.
Considering the natural projection map to the Euclidean plane $\pi: \mathbb{X}^3 \to \mathbb{E}^2$, $(x,y,z) \mapsto (x,y,0)$, 
we start by organizing the components in different sets.

 \begin{definition}\label{def:2-1} 
 Let $\Gamma = (\Gamma_1, \cdots, \Gamma_N)$ be a set of infinite colored geodesics embedded in $\mathbb{E}^2$, 
 which belong to $N \geq 2$ disjoint color groups such that, two geodesics belong to the \textit{same color group} if they do not intersect, 
or belong to \textit{different color groups} if they intersect once. 
 \end{definition}
 
Then, these geodesics will be lifted in the three-dimensional ambient space such that the infinite curve components do not intersect but cross over or under each other, as for general knots and links. 
This can be represented by two intersecting arcs on $\mathbb{E}^2$, after projection by the map $\pi$, as illustrated in Fig.~\ref{fig1}(B).
So, if $Q$ is such a point of intersection in the plane, then the inverse image $\pi^{-1}(Q)$ of $Q$ in $\mathbb{X}^3$ has exactly two distinct points.
We therefore mimic knot theory by recording the extra information of which arc is over and which is under 
at each such intersection point on the plane, and also call this region a \textit{crossing}, as explained in \cite{ref17}. 
Before lifting our geodesics to the thickened Euclidean plane, we describe the crossings in $\mathbb{E}^2$ using a sequence of positive integers, 
as suggested by Yaghi et al. \cite{15}, that will specify the number of consecutive over and undercrossings for each geodesic.

\begin{definition}\label{def:2-2} 
Let $i$, $j$, $k$, $l$ be strictly positive integers, and let $\Gamma_i$ and $ \Gamma_j$ be two disjoint sets of colored geodesics in $\mathbb{E}^2$.
Then by walking on an oriented geodesic $\gamma^k_i \in \Gamma_i$, the \textit{crossing sequence} $C^k_{i,j}$ of $\gamma^k_i$ with $\Gamma_j$ is defined either by, 
 \begin{enumerate}
    \item a sequence $(+1,0)$ (resp. $(0,-1)$) if $\gamma^k_i$ is always over (resp. under) all the components of $\Gamma_j$.
    \item a sequence $(\cdots, + p_l, - p_{l+1}, +p_{l+2}, \cdots)$ of minimal length, where $p_l$ are strictly positive integers, 
    such that there exists a crossing $c$ between $\gamma^k_i$ and $\gamma^{k_i}_j \in \Gamma_j$ whose closest neighboring 
    crossing in the opposite direction is an undercrossing, and from which $\gamma^k_i$ will have $p_l$ consecutive 
    overcrossings with the geodesics of $\Gamma_j$, followed by $p_{l+1}$ consecutive undercrossings, 
    followed by $p_{l+2}$ consecutive overercrossings and so forth. 
 \end{enumerate}
 Moreover, we denote by $\Sigma_{i,j} = (C^k_{i,j})_{k > 0}$, with $i,j  \in (1, ..., N), i < j$, the set of crossing sequences associated to the pair $(\Gamma_i,\Gamma_j)$, 
 $k$ being the index of the geodesics of $\Gamma_i$.
\end{definition}

\begin{remark}\label{rmk:2-3} 
The set of crossing sequences $\Sigma_{j,i} = (C^{k'}_{j,i}) _{k' > 0}$, with $i,j  \in (1, ..., N), i < j$ is deduced from $\Sigma_{i,j}$ for any pair $(\Gamma_i,\Gamma_j)$, and conversely.
\end{remark}

We can now define an \textit{untwisted weave} as the lift to the thickened Euclidean plane of these crossed planar geodesics, with respect to a set of crossing sequences.

\begin{definition}\label{def:2-4} 
Let $\Gamma = (\Gamma_1, \cdots, \Gamma_N)$ be a set of geodesics belonging to $N \geq 2$ color groups in $\mathbb{E}^2$ 
with a crossing information at each vertex according to a set of crossing sequences $\Sigma = (\Sigma_{i,j})_{i>j}$ with $i,j  \in (1, ..., N)$.
Then, we call \textit{untwisted weave} the lift to $\mathbb{E}^3$ by $\pi^{-1}$ of the pair $(\Gamma,\Sigma)$, which is an embedding of non-intersecting infinite curves in the thickened Euclidean plane. 
Each lifted geodesic is called a \textit{thread} and two threads are said to be in the same \textit{set of thread} $T_i$ if they are the lift of geodesics belonging to the same color group $\Gamma_i$. 
Moreover, we call \textit{strand} any compact non-degenerate subset $s \subseteq t$ of a thread $t$.
\end{definition}

\begin{remark}\label{rmk:2-5} 
The crossing sequence $C^k_{i,j}$ of a geodesic $\gamma^k_i$ with the set of colored geodesic $\Gamma_j$ 
is also the crossing sequence of its lift $t^k_i = \pi^{-1}(\gamma^k_i)$ with the set of threads $T_j$. 
We therefore use the same notation for both spaces. 
\end{remark}

Now, consider that given an untwisted weave, we allow continuous deformations, or isotopy, of our threads in $\mathbb{X}^3$. 
We have to characterize the property of such a weave to be \textit{entangled} by defining any non-entangled weave as the \textit{trivial weave}, also called \textit{unweave}. 
This means that the structure does not \textit{hang together}, as defined in \cite{ref11}. 

\begin{definition}\label{def:2-6} 
An untwisted weave is said to be the \textit{unweave} if it is isotopic to a single set of threads in $\mathbb{X}^3$.
\end{definition}

\begin{definition}\label{def:2-7} 
An untwisted weave $\mathrm{W}$ is said to be \textit{entangled}, if it is not isotopic to the unweave.
\end{definition}

Then, recall that a weave $\mathrm{W}$ in a general position in $\mathbb{X}^3$ can be projected onto the Euclidean plane by the map $\pi$, as in knot theory \cite{ref17}.   
By general position, we mean that the projection of two threads by $\pi$ in $\mathbb{E}^2$ are distinct,
This projection leads to a planar quadrivalent connected graph $W_0$, meaning that all the vertices have a degree four, which is isotopic to the original set of geodesic $\Gamma$.
In the particular case of a doubly periodic structure \cite{ref4}, that will be study in the next section, 
any unit cell can be seen as a link in the thickened torus as described in Fig.~\ref{fig2}. 

\begin{definition}\label{def:2-8} 
The projection $W_0$ of a weave $\mathrm{W}$ in general position onto $\mathbb{X}^2$ by the map $\pi: \mathbb{X}^3 \to \mathbb{X}^2$, $(x,y,z) \mapsto (x,y,0)$ 
is called a \textit{regular projection}, and once an over or under information is given at each vertex of $W_0$, 
we say that this structure is an infinite \textit{weaving diagram} $D_{W_0}$. 
Moreover, if $D_{W_0}$ is doubly periodic, then any unit cell contains essential simple closed curve components and is called a \textit{weaving motif}.
\end{definition}

\begin{figure}[h]
\centering
   \includegraphics[scale=1]{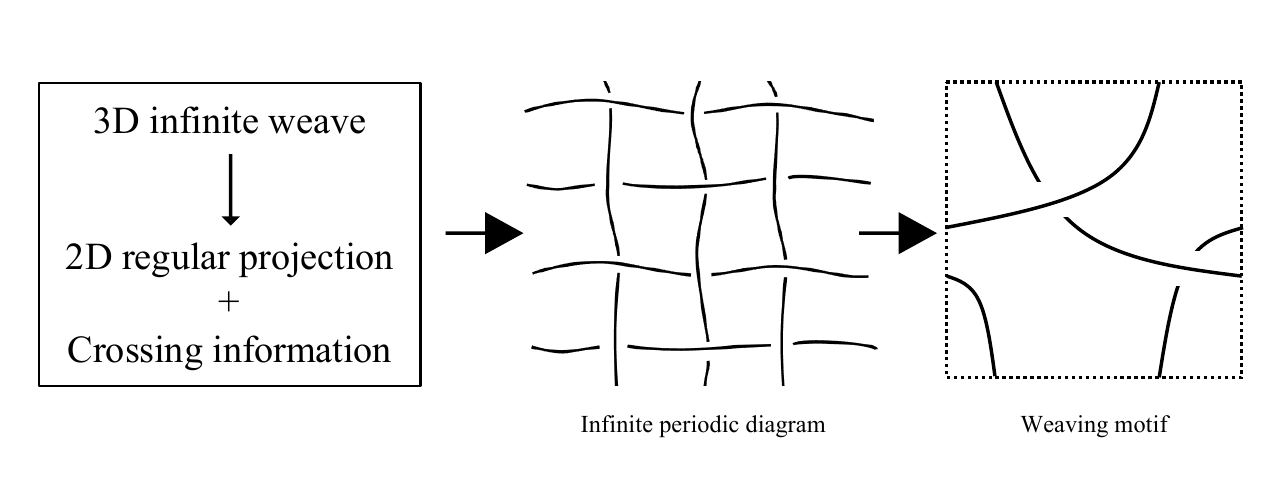}
      \caption{\label{fig2} Weaving Diagram}
\end{figure}

We can discuss the entangled property of a weave by analyzing its planar diagram.
First, we introduce the concept of \textit{blocking crossings} in terms of Reidemeister moves of type $\Omega_3$ (Fig.~\ref{fig7}).

\begin{definition}\label{def:2-9} 
Let $t_i \in T_i$, $t_j \in T_j$ and $t_k \in T_k$ be three threads of an untwisted weave $\mathrm{W}$, with $N \geq 2$ sets of threads $(T_1, \cdots,T_N)$, for all $j,k \in (1, \cdots , N)$ distinct. 
Then, if there exists a crossing $c = \pi(t_j) \cap \pi(t_k)$ on the weaving diagram $D_{W_0}$ of $\mathrm{W}$, 
such that a Reidemeister move $\Omega_3$ is not admissible for $\pi(t_i)$ at the neighborhood of $c$, we say that $c$ is a \textit{blocking crossing}.
\end{definition}

\begin{figure}[h]
\centering
   \includegraphics[scale=1]{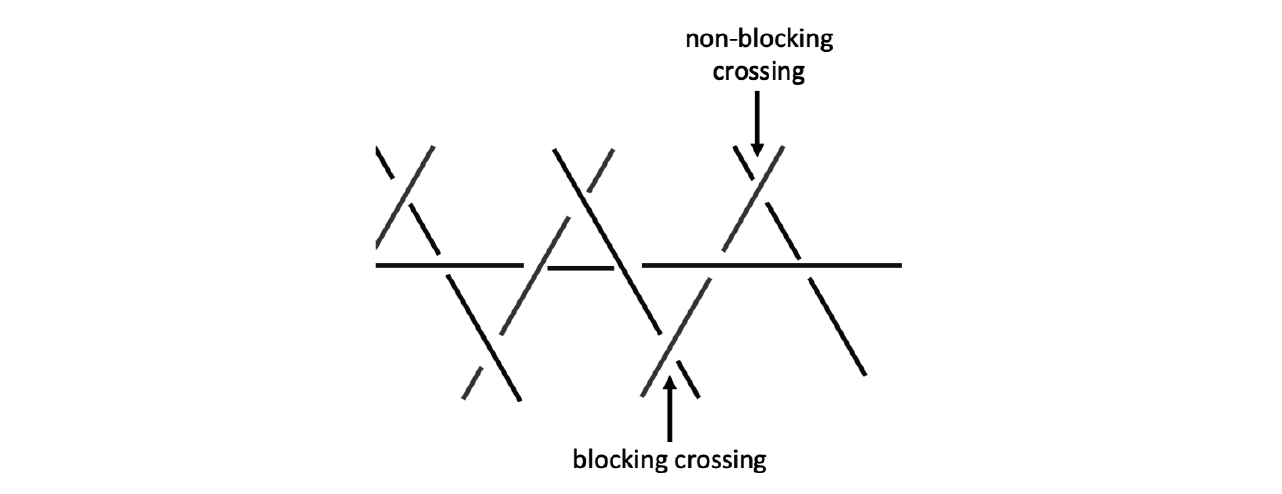}
      \caption{\label{fig3} Blocking crossing}
\end{figure}

Now, notice that any thread projected on $\mathbb{E}^2$ divides the plane in two disjoint regions, namely its \textit{left} and its \textit{right}, arbitrary labelled. 
We can characterize the entangled property of a weave $\mathrm{W}$ in terms of existence of blocking crossings in its diagram by the following Proposition, whose proof follows directly.  

\begin{proposition}\label{prop:2-10} 
An untwisted weave with $N \geq 2$ sets of threads $T_1, \cdots,T_N$ is \textit{entangled} if and only if 
for all $i \in (1, ..., N)$, each thread $t_i \in T_i$ projected on $\mathbb{E}^2$ admits a blocking crossing $c = \pi(t_j) \cap \pi(t_k)$ on its left, 
and a blocking crossing crossing $c' = \pi(t'_j) \cap \pi(t'_k)$ on its right, where $t_j , t'_j \in T_j$ and $t_k, t'_k \in T_k$ are disjoint threads, for all $j,k \in (1, \cdots , N)$ distinct. 
\end{proposition}

Finally, notice that the definition of an untwisted weave implies the existence of a twisted version. 
We introduce the definition of a \textit{twisted weave} below but restrict to the classification of untwisted weaves in this paper. 
The case of \textit{twisted weaves} is under study and will appear in a future work.\\
   
First, we start with the same sets of colored geodesics and crossing sequences $(\Gamma,\Sigma)$ in $\mathbb{E}^2$, as in Definition 2.4. 
However, before lifting this fixed pair to the Euclidean thickened plane, we need to introduce some local transformations in our sets of colored geodesics, 
which consists in possibly many local \textit{cut and glue} operations applied to two closest neighboring geodesics $\gamma$ and $\gamma'$ of a same color group,  
in order to introduce new crossings, called \textit{twists}, as illustrated in Fig.~\ref{fig4}. 

\begin{figure}[h]
\centering
   \includegraphics[scale=1]{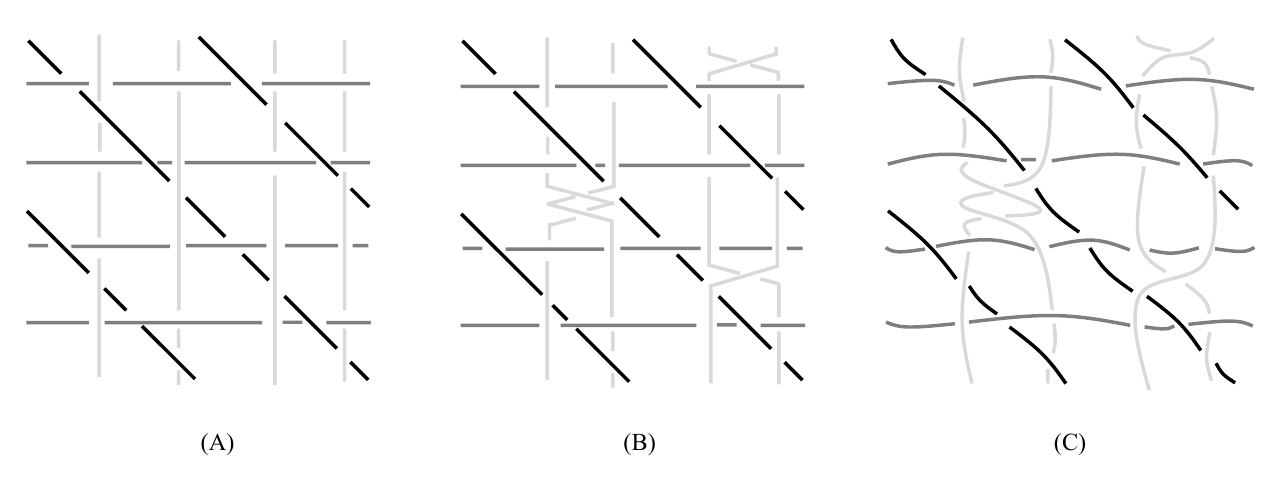}
      \caption{\label{fig4} 
(A) Set of colored geodesics with crossing information in the Euclidean plane. 
(B) Introduction of twists in the Euclidean plane. 
(C) Twisted weave in the thickened Euclidean plane.}
\end{figure}

These local twists can be characterized by the \textit{local linking number} $L_D(\gamma,\gamma')$ of the two twisted curves  $\gamma$ and $\gamma'$, 
which is the sum of the signs of all the crossings between these two curves in a given neighborhood $D$ of $\mathbb{E}^2$, where each crossing is given a sign, as in Fig.~\ref{fig5}. 

\begin{figure}[h]
\centering
   \includegraphics[scale=0.5]{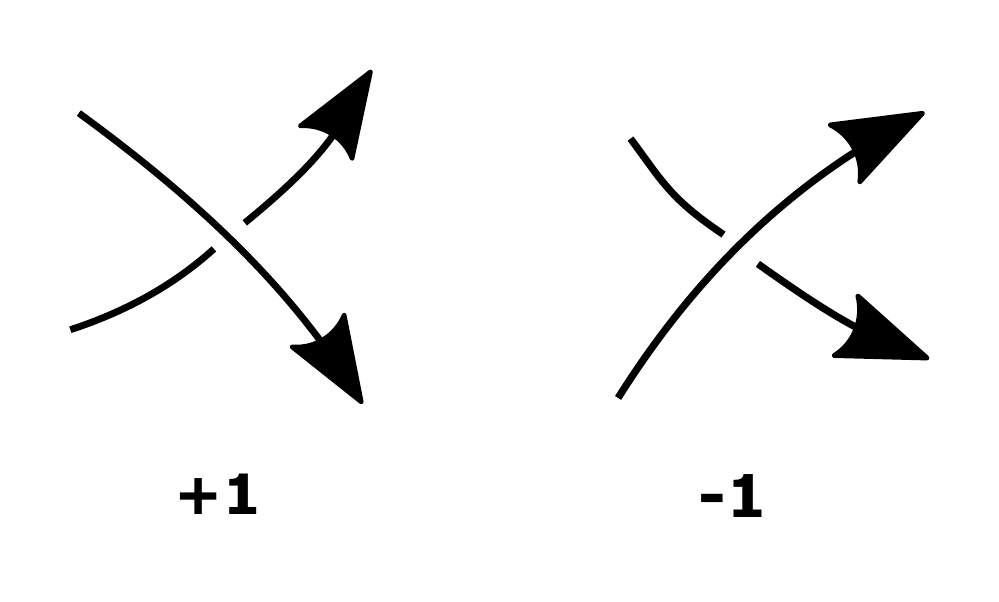}
      \caption{\label{fig5} Sign convention.}
\end{figure}

\begin{definition}\label{def:2-11} 
Let $\Gamma = (\Gamma_1, \cdots, \Gamma_N)$ be a set of geodesics belonging to $N \geq 2$ color groups in $\mathbb{E}^2$,  
and let $\Gamma_i$ be a set of geodesics of the same color $(\cdots, \gamma_i^{k-1}, \gamma_i^k, \gamma_i^{k+1},\cdots)$, 
indexed by a positive integer $k$ in terms of closest neighboring components, for any $i \in 1, \cdots , N$.
Let $D$ be a small disk containing only $\gamma_i^k$ and $\gamma_i^{k+1}$, meaning that there exists no crossing or any other geodesic in $D$.
Then, we say that $\gamma_i^k$ and $\gamma_i^{k+1}$ \textit{twist} if they are cut and glued such that $L_D(\gamma_i^k , \gamma_i^{k+1}) \neq 0$ in $D$, up to isotopy.
\end{definition}

Considering that the total linking number of two curves that twist in $n$ disjoint disks $D_1, \cdots D_n$ is the sum of the local linking number for each of these disks, 
we can define a weave from any pair $(\Gamma,\Sigma)$ that can be lifted into an untwisted weave, to which twists can be added such that the color groups associated to each set of threads is preserved.

\begin{definition}\label{def:2-12} 
A \textit{weave} is the lift to $\mathbb{X}^3$ of a pair $(\Gamma,\Sigma)$ satisfying Definition 2.4, 
possibly transformed to admit twists such that if two threads twist, they cannot twist with any other threads and their total linking number must be even. \\
\end{definition}


\section{Classification of doubly periodic untwisted weaves}\label{sec:3}

Weaves are mainly characterized by a number $N$ of disjoint sets of threads $T_1, …,T_N$, 
as well as a set of crossing sequences $\Sigma = \{\Sigma_{i,j} \ \ | \ \  i,j  \in (1, ..., N), i < j\}$.
As seen in the previous section, it is convenient to study their planar diagrams. 
Given a graph $\Gamma$ satisfying the definition of a regular projection, or equivalently Definition 2.1, and a set $\Sigma$, 
we have seen that it is possible to build a weaving diagram by assigning an over or under information to each vertex.
However, this attribution is not unique, as illustrated by some examples of existing woven materials showing that 
the weaving diagrams of two distinct weaves can be reconstruct from the same pair $(\Gamma,\Sigma)$.
The simplest cases are the diagrams related to the \textit{basket weave $(2,2)$} and the \textit{twill weave $(2,2)$}, see Fig.~\ref{fig6}. 
The two diagrams can be reconstructed from the skeleton of a square \textit{topological tiling}, which encodes the two sets of threads of these weaves, 
and such that each thread of the two sets has periodically two consecutive overcrossings followed by two consecutive undercrossings.
Nevertheless, these two woven materials have distinct physical properties, such as strength or stiffness, 
and it is important to characterize these differences from a mathematical point of view. 
Note that the notion of topological tiling considered here was defined by  B. Grünbaum and G.C. Shephard in \textit{Tilings and Patterns} (Chapter 4) \cite{ref10}, 
such that from a topological point of view, a tiling by irregular polygons is equivalent to a tiling by regular polygons 
if the application of a homeomorphism to such a tiling preserves the degree of the vertices and the number of adjacent and neighbors of each tile.
This observation motivated the study of equivalence classes of weaves and the development of a new parameter $\Pi$, 
such that any weaving diagram constructed from the triple $(\Gamma,\Sigma,\Pi)$ would be \textit{unique}, up to equivalence. 
In this paper, we will study the case of doubly periodic untwisted $(p,q)$-weaves and a generalization is planned for future work.

\begin{figure}[h]
\centering
   \includegraphics[scale=0.8]{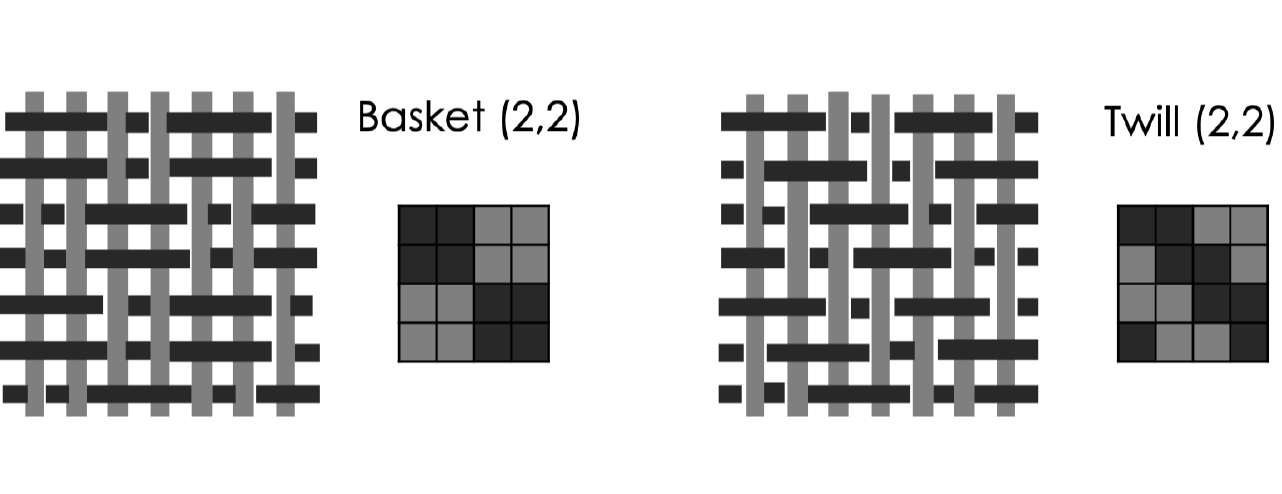}
      \caption{\label{fig6} Twill and Basket $(2,2)$ square weaving diagram with their associate design.}
\end{figure}

\subsection{Equivalence classes of doubly periodic untwisted $(p,q)$-weaves}\label{sec:3-1}

~

~

A doubly periodic untwisted $(p,q)$-weave $\mathrm{W}$ is an untwisted weave which admits translational symmetry in two non-parallel directions, 
and such that all the threads of the set $T_i$ have the same crossing sequence $C_{i,j}$ with the set $T_j$, described by two positive integers only.  
More precisely, recall from Definition 1.1 that considering the strictly positive integers $i$, $j$, $p$, and $q$, 
the crossing sequence $C_{i,j} = (+p,-q)$ associated with the disjoint sets of threads $T_i$ and $T_j$ 
is defined as if one travels along any thread $t_i \in T_i$, there exists a crossing $c_{i,j} = \pi(t_i) \cap \pi(t_j)$ having at least one of its two neighboring crossings
 $c'_{i,j}$ with a different over or under information, where $c'_{i,j}$ is a crossing between $t_i$ and another thread of $T_j$.
Then, walking on $t_i$ in the opposite direction of $c'_{i,j}$, there are cyclically $p$ crossings for which $t_i$ is over the other threads of  $T_j$, followed by $q$ crossings for which it is under. 
Moreover, if $C_{i,j} = (+p,-q)$, then $C_{j,i} = (+q,-p)$ is called the \textit{complementary crossing sequence} of $C_{i,j}$, 
and only one of these two crossing sequences is sufficient to describe the entanglement of the structure.
The case of non-crossing sets of threads with crossing sequences $(+1,0)$ or $(0,-1)$ is also considered for an entangled weave with $N > 2$, sets of threads.
Proposition 2.9 implies that on a periodic weaving motif, if a set of threads crosses another set, then there must be at least an overcrossing and an undercrossing 
between the threads of these two sets; otherwise, a single crossing may suffice if it is a blocking crossing for another set. 

The notion of equivalence adopted for this class of weaves has been defined by S. Grishanov et al. by an extension of 
the classical \textit{Reidemeister Theorem} for the general case of doubly periodic entangled structures represented by a torus diagram \cite{ref4}, namely a weaving motif in our case.

\begin{theorem}\label{th:3-1} \textbf{(Reidemeister Theorem for Weaves \cite{ref4})} 
Two doubly periodic weaves in $\mathbb{X}^3$ are \textit{ambient isotopic} if and only if their torus diagrams can be obtained from each other by a sequence 
of \textit{Reidemeister moves} $\Omega_1$, $\Omega_2$, and $\Omega_3$, isotopies on the surface of the torus, and \textit{torus twists}. 
\end{theorem}

\begin{figure}[h]
\centering
   \includegraphics[scale=0.65]{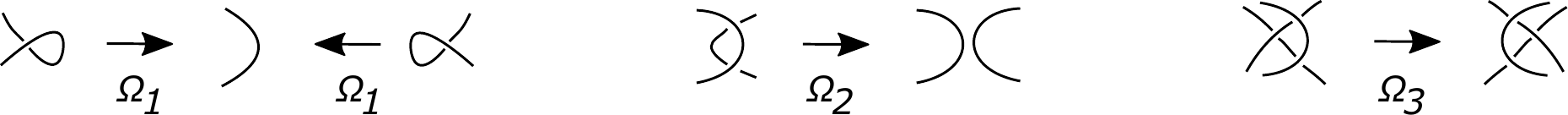}
      \caption{\label{fig7} Reidemeister moves}
\end{figure}


To construct our new parameter $\Pi$, we were inspired by the concept of \textit{design}, also called \textit{armor} in the literature, used for the classification of weaves in the textile industry \cite{ref18}.  
A design characterizes on a square grid the organization of crossings on a weaving motif, according to the periodic crossing sequence of an untwisted weave with $N = 2$ sets of threads. 
More information on designs from a mathematical viewpoint can be found in the works of Grünbaum and Shephard in \cite{ref7,ref8,ref9,ref11},   
which recall the results of E. Lucas \cite{ref2} who studied a particular class of doubly periodic untwisted weaves with two sets of threads, called regular satins, using arithmetic arguments. 
An example is given in Fig.~\ref{fig6} for the basket and twill weaves mentioned above.
To read a design, consider that the rows represent the horizontal strands on a unit cell, the columns the vertical strands, and each square represents a crossing.
Here, a gray square corresponds to a crossing where the horizontal strand is under the vertical strand, and conversely for a black square. 
Notice in Fig.~\ref{fig6} that the gray and black squares have a \textit{diagonal} organization for the case of twill weave, 
while they have a \textit{bloc} organization for the case of basket weave.
This illustrates the fact that there are different possibilities of assigning the crossing information to a same graph $\Gamma$, given a (set of) crossing sequence(s).
These different types of organization of crossings on a weaving motif justify the different physical properties of the corresponding woven materials.

Our purpose is to generalize this concept to doubly periodic untwisted $(p,q)$-weaving diagrams with $N \geq 2$ sets of threads, 
so that it can describe the organization of crossings for each pair of sets of threads on a unit cell.
Grünbaum and Shephard made an attempt using tilings by polygons in \cite{ref9}, 
which satisfy our definition of regular projection, considering geometric restrictions.
They assigned a label to each vertex in order to indicate the crossing information, 
but they mentioned that such an extended design can easily become very complicated and unintelligible. 
Our approach consist in creating a set of \textit{crossing matrices} associated with a weaving motif which would make it possible to distinguish the structures 
characterized by the same pair $(\Gamma,\Sigma)$, and thus become a \textit{weaving invariant} for the infinite diagram $D_{W_0}$.
Our concept of crossing matrices is directly related to the crossing sequences of a weaving diagram $D_{W_0}$, 
which means that each matrix is associated with a pair of distinct sets of threads of the diagrams. 
The elements of a crossing matrix are the symbols $+1$ representing an overcrossing, or $-1$ representing an undercrossing. 
In particular, we will work with square matrices associated with such a pair $(T_i , T_j)$, having size $m = p + q$, 
and associated to the crossing sequence $C_{i,j} = (+p , -q)$, as a natural generalization of Lucas' results.
This implies that a weaving motif $D_W$ will contain $m$ representative essential simple closed curves for each set $T_i$ and $T_j$.
We will see that once again only one matrix for each pair will suffice, so for a weaving diagram with $N$ sets of threads, 
we can define $\frac{N(N-1)}{2}$ matrices for $D_W$, 
such that each matrix encodes the crossing configuration between two sets of threads, from the point of view of one of them, see Fig.~\ref{fig8}.
This means that at an arbitrary crossing between two strands $s_i \in T_i$ and $s_j \in T_j$, with $T_i$ and $T_j$ two disjoint sets of threads, 
$s_i$ is over (resp. under) $s_j$, if we analyze the position of the strands of $T_i$ with respect to the strands of $T_j$.
or conversely that $s_j$ is under (resp. over) $s_i$, if we analyze the position of the strands of $T_j$ with respect to the strands of $T_i$. 

\begin{definition}\label{def:3-2} 
Let $i,j \in  \{1, \cdots , N\}$ and $C_{i,j} = (+p , -q)$ be the crossing sequence of the two disjoint sets of threads $T_i$ and $T_j$ 
of a doubly periodic untwisted $(p,q)$-weaving diagram $D_{W_0}$ with $N \geq 2$.
Let $M_{i,j}$ be a square $m \times m$ $(-1,+1)$-matrix, where $m = p + q$ is called the \textit{module} of $M_{i,j}$.
Then, $M_{i,j}$ is called the \textit{crossing matrix} of $D_{W_0}$ associated with $C_{i,j}$ if each row and each column of $M_{i,j}$ 
simultaneously contains $p$ consecutive symbol $+1$ followed by $q$ consecutive symbols $-1$, considering cyclic and countercyclical permutations of the rows and columns of the matrix. 
\end{definition}

\begin{remark}\label{rmk:3-3} 
To construct the crossing matrices associated to a fixed weaving motif $D_W$ of $D_{W_0}$ with $N \geq 2$ sets of threads, 
one must consider a flat torus cut from any preferred meridian-longitude pair, or equivalently fix any square unit cell of $D_{W_0}$ 
containing $m^2$ crossings and $m$ strands of each set. 
The strands $s_{k,i}$ of a set $T_i$ are oriented and indexed by a strictly positive integer $k$ from left to right and top to bottom, starting from the top left crossing.
Then, for any crossing sequence $C_{i,j}$ of $D_{W_0}$, one must use the following convention to fill the associated crossing matrix $M_{i,j} = (m_{x,y})_{0 \leq x, y \leq m-1}$,
 \begin{enumerate} 
    \item $m_{1,1} = +1$  (resp. $m_{1,1} = -1$) if the most top left crossing $c_{1,1} = s_{1,i} \cap s_{1,j}$ of $D_W$ 
    is such that the strand $s_{1,i}$ of $T_i$ is over (resp. under) the strand $s_{1,j}$ of $T_j$;
    \item fill the first row of the matrix by walking on $s_{1,i}$ such that the element $m_{1,k}$ gives the information of the crossing 
    $c_{1,k} = s_{1,i} \cap s_{k,j}$, with $k$ the index of the strand of $T_j$;
    \item fill the $k^{th}$ row of the matrix by walking on $s_{k,i}$, starting with the element $m_{k,1} = m_{1,k}$ given by the information 
    of the crossing $c_{k,1} = s_{k,i} \cap s_{1,j}$. Note that the $k^{th}$ column of the matrix is equal to its $k^{th}$ row.
 \end{enumerate}  
Finally, note that a cyclic or countercyclical permutation of the rows (resp. the columns) of such a crossing matrix, also called a translation of the design in \cite{ref11}, 
corresponds to a vertical (resp. horizontal) translation of the unit cell on $D_{W_0}$.
\end{remark}

\begin{figure}[h]
\centering
   \includegraphics[scale=1]{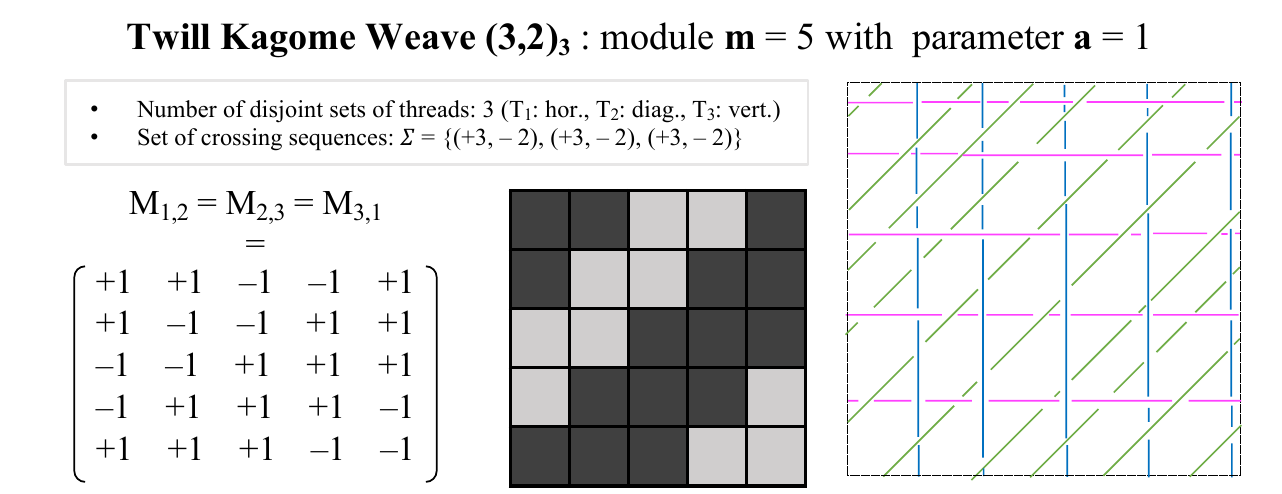}
      \caption{\label{fig8} Kagome Matrices}
\end{figure}

\begin{remark}\label{rmk:3-4} 
The information at the crossing between the threads $s_i$ and $s_j$ must be opposite on the two crossing matrices $M_{i,j}$ and $M_{j,i}$, 
describing the position of the two sets of threads considered, since if $s_i$ is over $s_j$ at a crossing, then the corresponding element of the matrix 
from the viewpoint of $T_i$ will be $+1$, while it will be $-1$ from the point of view of $T_j$, 
which explains why only one of the two matrices associated to a pair of sets of threads is enough to characterize our structures.
The crossing matrix $M_{j/i}$ is in fact directly deduced by transposing $M_{i/j}$ and changing the $+1$ values into $-1$ and conversely, $M_{j/i} = -(M_{i/j})^T$.
\end{remark}

As Lucas \cite{ref2}, we have attempted to define construction rules for crossing matrices of particular doubly periodic untwisted $(p,q)$-weaves, presented in the following examples. 
Note that the construction of a matrix describing a regular satin, reformulated with our notations, 
is also considered as an example for the general case of $N \geq 2$ sets of threads with at least a crossing sequence $(+p,-1)$. 
It would be interesting to enumerate all the possibilities to construct a crossing matrix, given any crossing sequence of type $(+p,-q)$ and to generalize to more complex periodic cases.

\begin{example}\label{ex:3-5} 
Consider the crossing matrix $M = (m_{x,y})_{0 \leq x, y \leq m-1}$ such that $x$ represents the column indices of the matrix and $y$ its row indices.  
Then, if $M$ is associated with a crossing sequence $(+p , -1)$ such that it satisfies Lucas' condition on regular satin \cite{ref2}, 
the symbols $-1$ can be positioned at the element $m_{x,y}$ satisfying the system,

     $ \left \{
         \begin{array}{r c l}
          x   & = & k, \mbox{ with } k \in {0, 1, \cdots , p+q-1} \\
          y   & = & a k \ mod (m), \mbox{ with } a < m \mbox{ fixed and } gcd(a,m)=1.
         \end{array}
        \right .$

In this case, the elements that do not satisfy this system are symbols $+1$.
\end{example}

\begin{example}\label{ex:3-6} 
Consider the crossing matrix $M = (m_{x,y})_{0 \leq x, y \leq m-1}$ such that $x$ represents the column indices of the matrix and $y$ its row indices.  
Then, if $M$ is associated with a crossing sequence $(+p , -p)$ such that at least two rows or two columns are equal,
the symbols $+1$ can be positioned at the element $m_{x,y}$ satisfying the system,
    
     $ \left \{
         \begin{array}{r c l}
          x_k   & = & k, \mbox{ with } k \in \{0, 1, ... , 2p-1\} \\
          y_{k,l}   & = & y_{k,l-1} \pm 1, \mbox{ if } k \neq p \ mod (m), \mbox{ or } \\
          y_{k,j} & = & y_{k,j-1} \pm p, \mbox{ otherwise, } j \in \{1, \cdots , p-1\}.
         \end{array}
        \right .$
            
    with the first column given by,
    
     $ \left \{
         \begin{array}{r c l}
          x_0   & = & 0 \\
          y_{0,0}   & = & 0  \mbox{ and } y_{0,j} = y_{0,j-1} \pm 1, j \in \{1, \cdots , p-1\}.
         \end{array}
        \right .\\$

In this case, the elements that do not satisfy this system are symbols $-1$.
\end{example}

\begin{example}\label{ex:3-7} 
Consider the crossing matrix $M = (m_{x,y})_{0 \leq x, y \leq m-1}$ such that $x$ represents the column indices of the matrix and $y$ its row indices.  
Then, if $M$ is associated with a crossing sequence $(+p , -q)$, with $p,q \neq 1$ and such that there are not two equal rows or columns,
the symbols $+1$ can be positioned at the element $m_{x,y}$ satisfying the system, 

     $ \left \{
         \begin{array}{r c l}
          x   & = & k, \mbox{ with } k \in {0, 1, \cdots , p+q-1} \\
          y   & = & a k \pm i \ mod (m), \mbox{ with } a =  \pm 1 \mbox{ fixed, } i \in \{1, \cdots , p-1\}.
         \end{array}
        \right .\\$
        
In this case, the elements that do not satisfy this system are symbols $-1$.
\end{example}

\begin{remark}\label{rmk:3-8} 
Note that the basket weave of Fig.~\ref{fig6} has a crossing matrix of rank $1$, while the twill weave has a crossing matrix of rank $2$, 
which confirms the non-equivalence of the two structures. This result can be generalize for any crossing matrix corresponding to a crossing sequence $(+p,-p)$, 
whose rank would be equal to $1$ if it is constructed as in Example 3.6, or equal to $p$ if it is constructed as in Example 3.7.
\end{remark}

The introduction of crossing matrices makes it possible to distinguish two weaving diagrams characterized by the same pair $(\Gamma,\Sigma)$, 
by assigning them a fixed sequence of crossing matrices $\Pi = \{M_{i,j} \setminus i, j \in \{1, \cdots , N\}\}$, with $M_{i,j}$ the crossing matrix defined above, 
and such that the sets of threads are identically indexed in order to compare strands from a same set on the two weaving motifs.
Therefore, a triple $(\Gamma,\Sigma,\Pi)$ allows to characterize the equivalence classes of doubly periodic untwisted $(p,q)$-weaving diagrams of $\mathbb{E}^2$, 
with respect to the generalized Reidemeister Theorem (Theorem 3.1). We can state and prove our main Theorem of this subsection.

\begin{theorem}\label{thm:3-9} \textbf{(Equivalence Classes of doubly periodic untwisted $(p,q)$-weaves)} 
Let $\mathrm{W}_1$ and $\mathrm{W}_2$ be two doubly periodic untwisted $(p,q)$-weaves with $N \geq 2$ sets of threads,
such that their corresponding regular projections are equivalent, up to isotopy of $\mathbb{E}^2$, and with the same set of crossing sequences.
Let $D_{W_1}$ and $D_{W_2}$ be two weaving motifs of same area of $\mathrm{W}_1$ and $\mathrm{W}_2$, respectively.
Then, $D_{W_1}$ and $D_{W_2}$ are equivalent if and only if their crossing matrices are pairwise equivalent, 
meaning that all the matrices of $D_{W_2}$ can simultaneously be obtained from the respective matrices of $D_{W_1}$ from at least one of two conditions,  
 \begin{itemize} 
    \item a same cyclic or countercyclical permutations of the rows and the columns;
    \item a same clockwise or counterclockwise rotation of $\pi$, or of $\frac{\pi}{2}$ together with an inversion of all its symbols.
 \end{itemize}         
\end{theorem}

\begin{proof}
First, we prove that the equivalence of the weaving motifs implies the equivalence of the crossing matrices, using Theorem 3.1.
We start by studying the invariance by the Reidemeister moves.
Without loss of generality, we assume that our weaving diagrams are geodesic. 
Then, by definition, Reidemeister moves of types $\Omega_1$ and $\Omega_2$ do not occur. 
For a Reidemeister move of type $\Omega_3$, we consider three threads from three different sets.
However, recall that each crossing matrix is obtained from the crossing information of a pair of sets of threads only. 
This means that an $\Omega_3$ move does not change any of the crossing matrices. 
Now, it suffice to show the invariance of the matrices of a weaving motif $D_W$ for the torus twists. 
We fix a lattice $\mathbb{Z}^2$ in $\mathbb{E}^2$ where the infinite weaving diagrams are embedded.
Let $p$ be the intersection point of preferred meridian-longitude pair $(\mu ,\lambda)$ on $D_W$, such that a flat weaving motif is obtained by cutting along this pair. 
Notice at this point that if we reverse the meridian with the longitude, meaning that we replace the pair $(\mu ,\lambda)$ into $(\lambda,\mu)$, 
then the cut along this reversed pair is realized by rotating the original weaving motif in $\mathbb{E}^2$ by $\frac{\pi}{2}$.
Let $p' \neq p$ be a point on the longitude $\lambda$ and $p'' \neq p$ be a point on the meridian $\mu$ of $D_W$. 
Then, we obtain new preferred meridian-longitude pairs $(\mu ,\lambda')$ and $(\mu' ,\lambda)$ on $D_W$. 
Cutting $D_W$ along $(\mu' ,\lambda)$ instead of $(\mu ,\lambda)$ corresponds to a cyclical or countercyclical permutations of the rows of the crossing matrices of $D_W$.
Similarly, cutting $D_W$ along $(\mu ,\lambda')$ instead of $(\mu ,\lambda)$ corresponds to a cyclical or countercyclical permutations of the colums of the crossing matrices of $D_W$.
In particular, these transformations correspond to a vertical and horizontal translation of the unit cell in the infinite weaving diagram respectively.
Besides, recall that a torus twist represents a modular transformation which preserves the fixed lattice \cite{ref3}, 
and that two unit cells of same area of an infinite weaving diagram have the same number of crossings \cite{ref4}. 
Therefore, two weaving motifs related by a sequence of torus twists would correspond to two distinct parallelogram unit cells of the same infinite diagram. 
We can conclude using Remark 3.3 and the definition of the crossing sequence of a $(p,q)$-weaves, that the crossing matrices of two such diagrams 
are equivalent up to a sequence a cyclical or countercyclical permutations of the rows and columns, which conclude the first part of our proof.
Then, to prove the reverse implication, we start from a set of crossing matrices to which we apply a same transformation.
If we apply a same cyclic or countercyclical permutation of the rows and the columns applied simultaneously to all the crossing matrices associated to a weaving motif, 
we have seen above that it corresponds to a translation of the periodic unit cell in the infinite diagram, which implies the equivalence of the corresponding weaving motifs.  
Furthermore, if we apply to the given set of crossing matrices a same clockwise or counterclockwise rotation of $\pi$, or of $\frac{\pi}{2}$ together with an inversion of all its symbols, 
then we would obtain pairwise equivalent matrices, up to cyclic or countercyclical permutations of the rows and the columns, from which we can construct once again equivalent weaving motifs. 
\end{proof}

\begin{remark}\label{rmk:3-10} 
Note that if one of the crossing matrix of $D_{W_2}$ has a different permutation or rotation than the others matrices, then $D_{W_1}$ and $D_{W_2}$ are not equivalent.
Indeed, there would exist a thread $t_i$ of $D_{W_2}$ for which the order of all its crossings will be different than its representative in $D_{W_1}$. 
This means that there would exist a crossing $c_{i,j},$ in $D_{W_2}$, with $T_i$ and $T_j$ two sets of threads, such that by walking on the thread $t_i$, 
one of the nearest neighboring crossing $c_{i,k},$ of $c_{i,j},$, with a different set of threads $T_k$ will have a different type than the corresponding one in $D_{W_1}$. 
Thus, the two weaving diagrams cannot be superimposed and are therefore not equivalent, as in Fig.~\ref{fig9}.
\end{remark}

\begin{figure}[h]
\centering
   \includegraphics[scale=1.2]{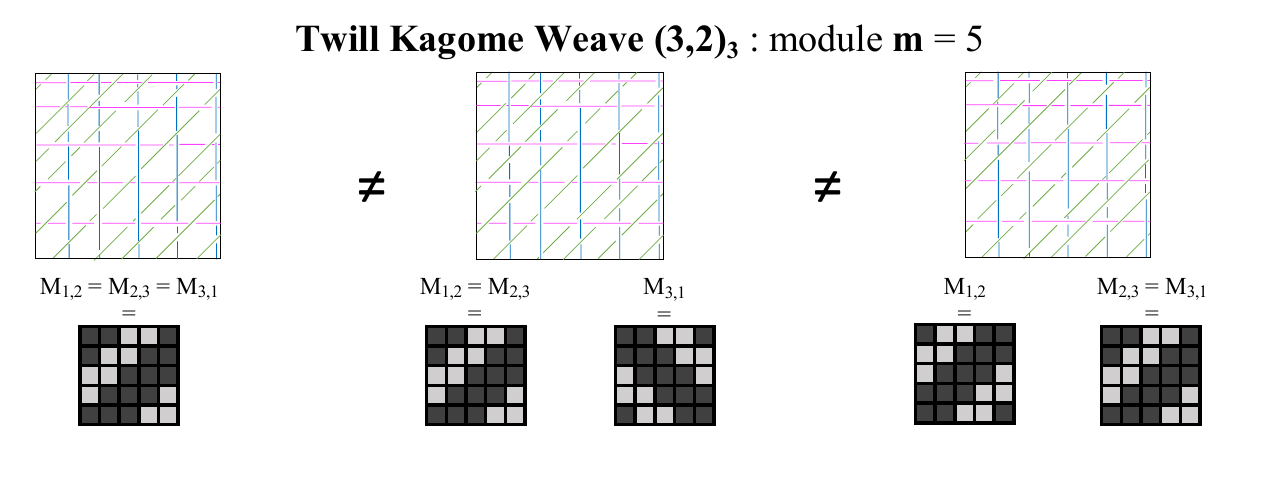}
      \caption{\label{fig9} Non Equivalent Kagome Matrices}
\end{figure}

One of the other great interests of these crossing matrices for the study of equivalence classes of weaving diagrams 
is that by defining any arbitrary weaving diagram by a triple $(\Gamma,\Sigma,\Pi)$, we can find non equivalent structures with the same pair $(\Gamma,\Sigma)$ 
just by doing a permutation or rotation of one or two of the matrices of $\Pi$, with respect the related crossing sequences, as described in Theorem 3.9. 
This is very useful for the construction an classification of our diagrams.\\

\subsection{Crossing number of doubly periodic untwisted $(p,q)$-weaves}\label{sec:3-2}

~

~

By definition of the projection $W_0$ of a doubly periodic untwisted $(p,q)$-weave $\mathrm{W}$ and of a crossing, 
we notice that $W_0$ is isotopic to the skeleton $\Gamma$ of a particular type of periodic topological quadrivalent tiling by convex polygons, that we call a  \textit{thread-tiling}.

\begin{definition}\label{def:3-11} 
A periodic \textit{thread-tiling}, composed of $N \geq 2$ sets of threads, is a planar edge-to-edge quadrivalent topological tiling 
by convex polygons, such that each edge of these polygons belongs to a single thread and two adjacent edges belong to two threads of different sets. 
\end{definition}

Now our objective is to characterize a doubly periodic thread-tiling using combinatorial arguments, 
which would be useful for the classification of our weaves by their crossing number, 
meaning the minimal number of crossing that can be found on an associated weaving motif. 
We call such a diagram a \textit{minimal diagram}.
As mentioned in \cite{ref4}, this is one of the main challenges for the classification of doubly periodic entangled structures, due to the many possibilities of choosing a unit cell. 
Our strategy consists in characterizing a doubly periodic thread-tiling on a periodic unit cell by a set $\Sigma'$ of essential simple 
closed curves embedded on a torus, each representing a set of threads, and generating the targeted graph $\Gamma$.
Recall that such a curve can be described by a pair of coprime integers $(a,b)$ \cite{ref3}, which embeds the slope of the geodesic thread in the Euclidean plane. 
Thus, considering that the couples $(a,b)$ and $(-a,-b)$ represent threads belonging to a same set, and therefore equivalent, as well as the pairs $(-a,b)$ and $(a,-b)$, 
we use the convention $(a,b)$ and $(-a,b)$ for these two cases, respectively.
Moreover, notice that the number of intersections $|v|$ between two of these curves $(a,b)$ and $(a',b')$ is given by the \textit{geometric intersection number} \cite{ref3}, 
defined by the equation $|a \ b' - a' \ b| = |v|$, which encodes the periodicity of the structure after translation in both directions.
This number is defined for free homotopy classes of simple closed curves and is the minimum number of intersections between a representative of each set of threads. 
Note that a same thread-tiling can be characterized by non-equivalent sets of curves, and the idea is to find a periodic unit cell with the fewest vertices. 
Moreover, a weave can admit different thread-tilings for projection by isotopy, if it admits a Reidemeister moves of type $\Omega_3$  for example. 
We believe that the classification of these doubly periodic thread-tilings in terms of sets of essential simple closed curves on a unit cell is an interesting problem, 
and we assume that for a given number of set of threads, a periodic unit cell of the tiling containing curves given by pairs of coprimes minimized in absolute value will have the fewest vertices.

To study the crossing number of doubly periodic untwisted $(p,q)$-weave, recall that each element of the pair $(\Sigma',\Sigma)$ related to a given set of threads 
indicates the periodic behaviors of every thread of this set. 
This means that for each pair of disjoint sets of threads, we can deduce from the associated geometric intersection number of the 
two representative curves and the crossing sequence the minimum number of crossings required to ensure the periodicity associated with these two sets.

\begin{lemma}\label{lem:3-12} \textbf{(Pairwise crossing number)}
Let $D_{W_0}$ be a weaving diagram with $N$ sets of threads characterizes by the pair $(\Sigma',\Sigma)$. 
Let $i, j \in \{1, \cdots , N\}$ be distinct integers and $T_i$, $T_j$ be two sets of threads of $D_{W_0}$. 
Then, to ensure the periodicity on a weaving motif of $D_{W_0}$, the minimum number of necessary crossings $c_{i,j}$ on a thread $t_i \in T_i$ and $t_j \in T_j$ is given by,
$$\mathcal{C}_{i,j} = lcm ( \zeta^i_{i,j} ;  \zeta^j_{i,j})$$
where $\zeta^i_{i,j} = |v_{i,j}| \times lcm(\{|C_{i,k}| \ / \ k \in \{1, \cdots , N\}\})$, with $|C_{i,k}|$ the sum of the two positive integers of the crossing sequence $C_{i,k}$, for all $k \in \{1, \cdots , N\}$.
We call $\mathcal{C}_{i,j}$ the $(i,j)$-pairwise crossing number. 
\end{lemma}

\begin{proof}
Let $s_i \in T_i$ and $s_j \in T_j$ be the two curves representative of a unit cell of the thread-tiling corresponding to $D_{W_0}$.
Then, to ensure the periodicity of the crossing sequence $C_{i,j}$ in a weaving motif of $D_{W_0}$
we must first consider the minimum number of vertices $v_{i,j}$ which is necessary on $s_i$ and $s_j$ to satisfy the definition of the thread-tiling given by $\Sigma'$. 
This number $|v_{i,j}|$ is the geometric intersection number between $s_i$ and $s_j$ by definition.
Then, we must multiply this number by the minimum number of crossings necessary to read the crossing sequence $C_{i,j} = (+p,-q)$, 
which is its module $m = |C_{i,j}| = p + q$, that is also minimal by definition.
However, on a strand $s_i$ (resp. $s_j$) there are not only crossings of type $c_{i,j}$ when $N > 2$, we must therefore also consider the minimum number of necessary crossings 
of type $c_{i,k}$ on $s_i$ (resp. $c_{j,k}$ on $s_j$) to read the other crossing sequences $C_{i,k}$ (resp. $C_{j,k}$) and ensure the periodicity.
We must therefore consider the least common multiple of their modules $|C_{i,k}|$ (resp. $|C_{j,k}|$) as multiplier of $|v_{i,j}|$.
Finally, the global minimality is ensured by taking the least common multiple of these two products, 
considering that we must obtain the same number of crossings on each thread of $T_i$ and $T_j$ on the weaving motif.
\end{proof}

\ ~
 
To fulfill the definition of a weaving motif, each $(i,j)$-pairwise crossing number $\mathcal{C}_{i,j}$ represent the number of crossing that must belong to the same representative 
simple closed curves for each set $T_i$ and $T_j$, unless these sets do not cross. In this case, at least one of the two curves must have $\mathcal{C}_{i,j}$ crossings.
However, this condition must be satisfied for all pairs of sets of threads in the weaving motifs, considering that each representative curve of a set of thread is 
characterized by a pair of coprime integers, which must not be equivalent to the pair representing another set. 
This leads to solving a system of equations in which one can simultaneously satisfy the geometric intersection number equations 
for each pair of disjoint sets of threads, as stated in our second main theorem.
Note that the possibilities of choosing the integers are restricted and that we might need more than one representative curves for a set of threads in some cases.

\begin{theorem}\label{thm:3-13} \textbf{(Total Crossing Number)}
Let $i, j \in \{1, \cdots , N\}$ distinct and $\mathcal{C}_{i,j}$ be the $(i,j)$-pairwise crossing numbers of a weaving diagram 
$D_{W_0}$ with $N$ sets of threads, characterized by the pair $(\Sigma',\Sigma)$. 
Let $(\mathcal{S}_{min})$, be the system of geometric intersection number equations, defined for integers $a_i$ and $b_i$, 
either coprime or such that one of them equal zero,  
satisfying for each equation that we can multiply both parts by a same even number $k_l$ 
if the two sets of threads implied on the equation cross ($k_l = 1$ otherwise), 
with $l \in \{1, \cdots , \frac{N (N - 1)}{2}\}$ being the index of the equation in the system.\\

 $ (\mathcal{S}_{min}) \left \{
   \begin{array}{r c l}
      \ k_1 \times |a_1 \ b_2 - a_2 \ b_1|   & = & k_1 \times \mathcal{C}_{1,2} \\
      . \\
      . \\
      . \\
      k_l \times |a_i \ b_j - a_j \ b_i|   & = & k_l \times \mathcal{C}_{i,j} \\
      . \\
      . \\
      .
   \end{array}
   \right .$
   
 ~ \\

Then, from the solution of $(\mathcal{S}_{min})$ with smallest multipliers $k_{l}$ which minimizes each integers $a_m$ and $b_m$ in absolute value, 
and such that every two pairs $(a_m , b_m)$ are distinct, we can deduce the total crossing number of $D_{W_0}$ given by,
$$ \mathcal{C} = \sum_{i<j = 1}^{N} \mathcal{C'}_{i,j} $$ 
with each $\mathcal{C'}_{i,j} = k_l \ \mathcal{C}_{i,j}$ in $(\mathcal{S}_{min})$.
\end{theorem}

\begin{proof}
A unit cell of a periodic weaving diagram is an embedding of essential simple closed curves crossing on the surface of a torus. 
Thus, the geometric number of intersection between a simple closed curve of a set of thread $T_i$ and a simple closed curve of a set $T_j$ 
defines the minimum number of crossings between these two curves on a weaving motif, given by the $(i,j)$-pairwise crossing number $\mathcal{C}_{i,j}$. 
From Lemma 3.12, we therefore have to find two pairs integers, either coprime or such that one of them equal zero,  
whose geometric number of intersection is $\mathcal{C}_{i,j}$, for each pair $(i,j)$.
$$|a_i \ b_j - a_j \ b_i| = \mathcal{C}_{i,j}$$
These four integers are chosen by convention such that their absolute value is minimal.
Moreover, since this condition must be satisfied simultaneously for all distinct pairs of sets of threads $(T_i,T_j)$, 
we conclude that to find the minimal number of crossings on the weaving motif, 
we have to solve the system of equations $(\mathcal{S}_{min})$ of the theorem statement.
\end{proof}

\ ~

\begin{remark}\label{rem:3-14} 
Considering that two periodic weaving motifs with same area contain the same number of crossings, 
we can construct a minimal diagram on a square unit cell given the fact that  $(|a| + |b|)$ parallel segments that do not intersect any corner of the square, 
correspond to a $(a,b)$-simple closed curve on the torus after identification of the opposite sides of the square.
Indeed, by taking $k_l$ (or $\frac{k_l}{2}$ if $k_l$ is an even number divided to be distributed for the four integers of the geometric intersection number equation) 
simple closed curves of same slope for each set of threads $T_i$, represented by the pairs $(a_i, b_i)$ and satisfying Theorem 3.13, 
we can construct a minimal diagram on a square unit cell, whose associated infinite weaving diagrams can be built from the pair $(\Sigma',\Sigma)$, up to isotopy.
\end{remark}

The classification tables of Section 3.3 show examples of minimal diagrams for some simple cases of weaving diagrams.
However, as discussed in the previous subsection, non-equivalent weaving diagrams can be constructed from a same pair $(\Sigma',\Sigma)$. 
At this point, we can conclude about their crossing number and their minimal diagram. 

\begin{remark}\label{rek:3-15} 
The total crossing number of a doubly periodic untwisted $(p,q)$-weave is given by the following.
 \begin{itemize} 
    \item If a set of crossing-matrices, corresponding to a minimal diagram constructed by Remark 3.14 is unique up to equivalence, then the associated weaving diagram 
    built from the pair $(\Sigma',\Sigma)$ is unique and has its crossing number given by Theorem 3.13.
    \item If a non-equivalent set of crossing matrices is found, it means that there exist at least two non-equivalent weaving diagrams generated from a same couple $(\Sigma',\Sigma)$.
    One of these diagrams has its crossing number given by Theorem 3.13 and a minimal unit cell constructed using Remark 3.14.
    The crossing numbers of the other non-equivalent weaving diagrams are given by solving the system $(\mathcal{S}_{min})$ of Theorem 3.13 with the next smallest solutions, 
    such that the corresponding weaving diagrams generated by Remark 3.14 have associated crossing-matrices that are not equivalent to the ones given for the previous smallest solutions of $(\mathcal{S}_{min})$.
 \end{itemize}    
\end{remark}

We can illustrate this method with the same example of basket and twill square (2,2) weaving diagrams, see also Figure 8. 

\begin{example}\label{ex:3-16} 
The minimal diagram for the twill case is obtained with one representative simple closed curve for each set of thread.
Recall that each of these curves must have four crossings, two over and two under, so its crossing number is $\mathcal{C}=4$.
This can be done with a $(2,1)$-curves for one set and a $(-2,1)$-curves for the other set.
For the basket case, we find that $k_{1,2} = 2$ is the next smallest solution of $(\mathcal{S}_{min})$, with the
four integers satisfying this relation and such that their absolute value is minimal, given by 
two $(1,1)$-curves for one set and two $(-1,1)$-curves for the other set.
Finally since we can organize the crossings on an associated unit cell with eight crossings, 
such that the corresponding crossing matrix is not equivalent to the one of the twill case,
we can confirm that the crossing number of the plain square (2,2) weaving diagram is eight.
\end{example}

A sub-classification for each pair of regular projection and set of crossing sequences by solving the system $(\mathcal{S}_{min})$ in Theorem 3.13, 
with different values for the multipliers $k_{l}$, starting from the smallest possible, and by studying the equivalence class of the set of crossing matrices 
can also be done for a better classification.\\

We end this paper by studying the invariance of the number of blocking crossings in equivalent minimal diagrams.
By definition of a blocking crossing (Definition 2.9) in the case of a weave with $N \geq 3$ sets of threads, 
we notice that it locally implies the existence of \textit{alternating triangles}, or \textit{A-triangle}, in the weaving diagram. 
Such an A-triangle is locally defined by three threads $t_i \in T_i$, $t_j \in T_j$ and $t_k \in T_k$, with $i, j, k \in \{1, \cdots , N\}$ distinct, 
and three closest neighboring crossings of different types, such that $t_i$ is over $t_j$ and under $t_k$, and $t_j$ is over $t_k$, or conversely up to the labels (see Fig. 3).
Recall that two minimal diagrams of same area of a given weave have the same number of crossings. 
Then, we can prove the following Proposition with the same strategy as Theorem 3.9.

\begin{proposition} [The A-triangles Number]\label{prop:3-16} 
Let $D_{W_1}$ and $D_{W_2}$ be two minimal weaving motifs of doubly periodic untwisted $(p,q)$-weaves with $N \geq 3$ sets of threads. 
Then, they have the same number of A-triangles if and only if every Reidemeister moves of type $\Omega_3$ that occur at a crossing belonging to a A-triangle 
cross the complete A-triangle.
\end{proposition}

\begin{proof}
The proof is almost similar to the one of Theorem 3.9. The difference concerns the Reidemeister move $\Omega_3$.
We denote by $K$ the number of A-triangles in $D_{W_1}$. 
If an $\Omega_3$ move happens at a crossing that does not belong to a A-triangle, then $K$ is obviously not affected.
However, if it occurs at a crossing belonging to a $A-triangle$, there are two possibilities.
First, the thread lies in the interior of the A-triangle after the $\Omega_3$ move. 
It generates a new triangle which is not alternating, meaning that $K$ decreases, which is a contradiction.
Second possibility, the thread crosses the A-triangle after the $\Omega_3$ move, meaning that it is in the exterior of the A-triangle. 
In this case, $K$ remains unchanged.
\end{proof}

\subsection{Examples of Classification tables}\label{sec:3-3}

~

~

The purpose of this section is to apply our new results to build and classify some simple square and kagome weaving diagrams by hand, 
which could be extended to more complex structures with a computer program.

\begin{figure}[H]
\centering
   \includegraphics[scale=0.45]{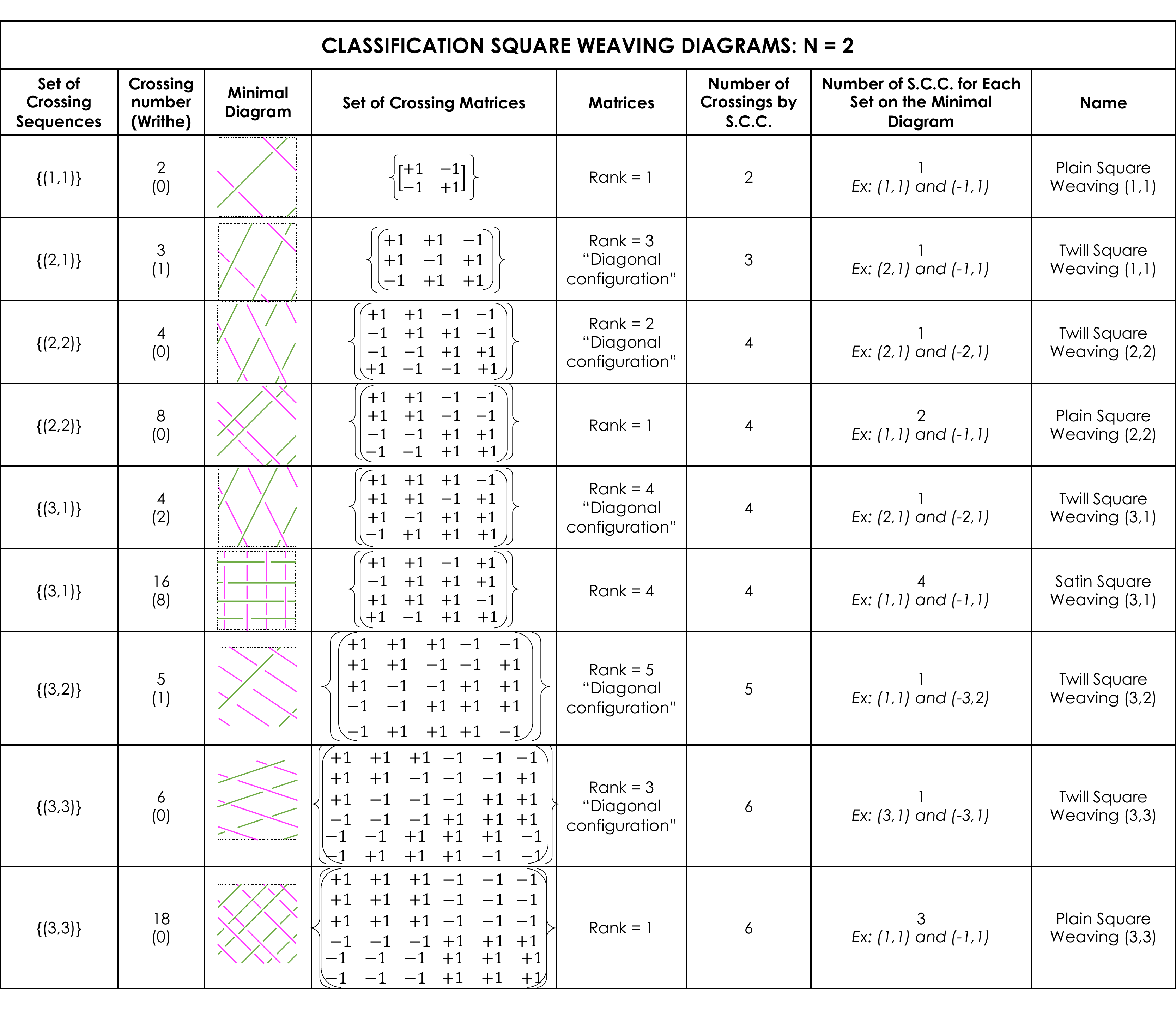}
      \caption{\label{fig10} Table Classification Square Weavings 1}
\end{figure}

\begin{figure}[H]
\centering
   \includegraphics[scale=0.45]{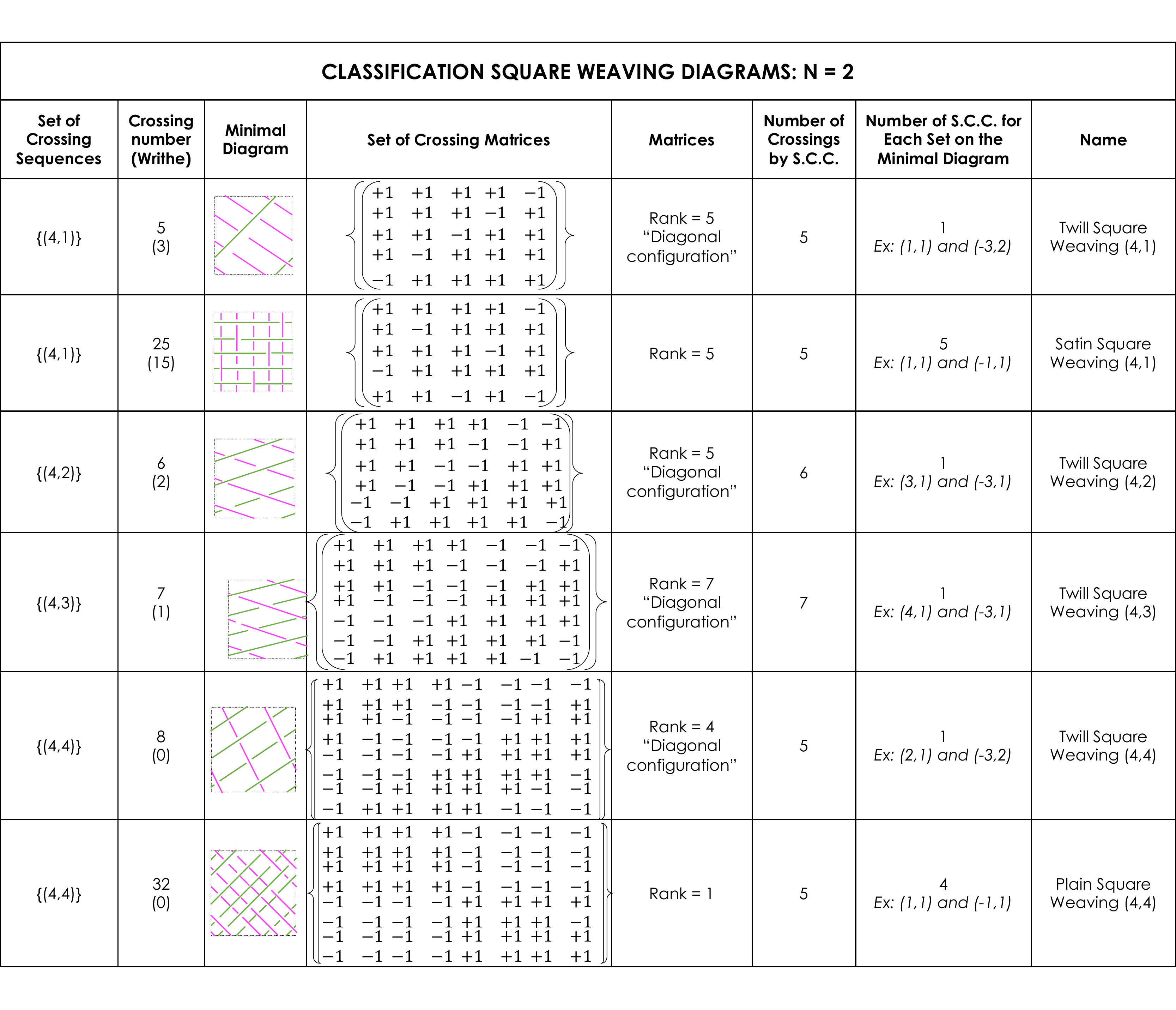}
      \caption{\label{fig11} Table Classification Square Weavings 2}
\end{figure}

\begin{figure}[H]
\centering
   \includegraphics[scale=0.47]{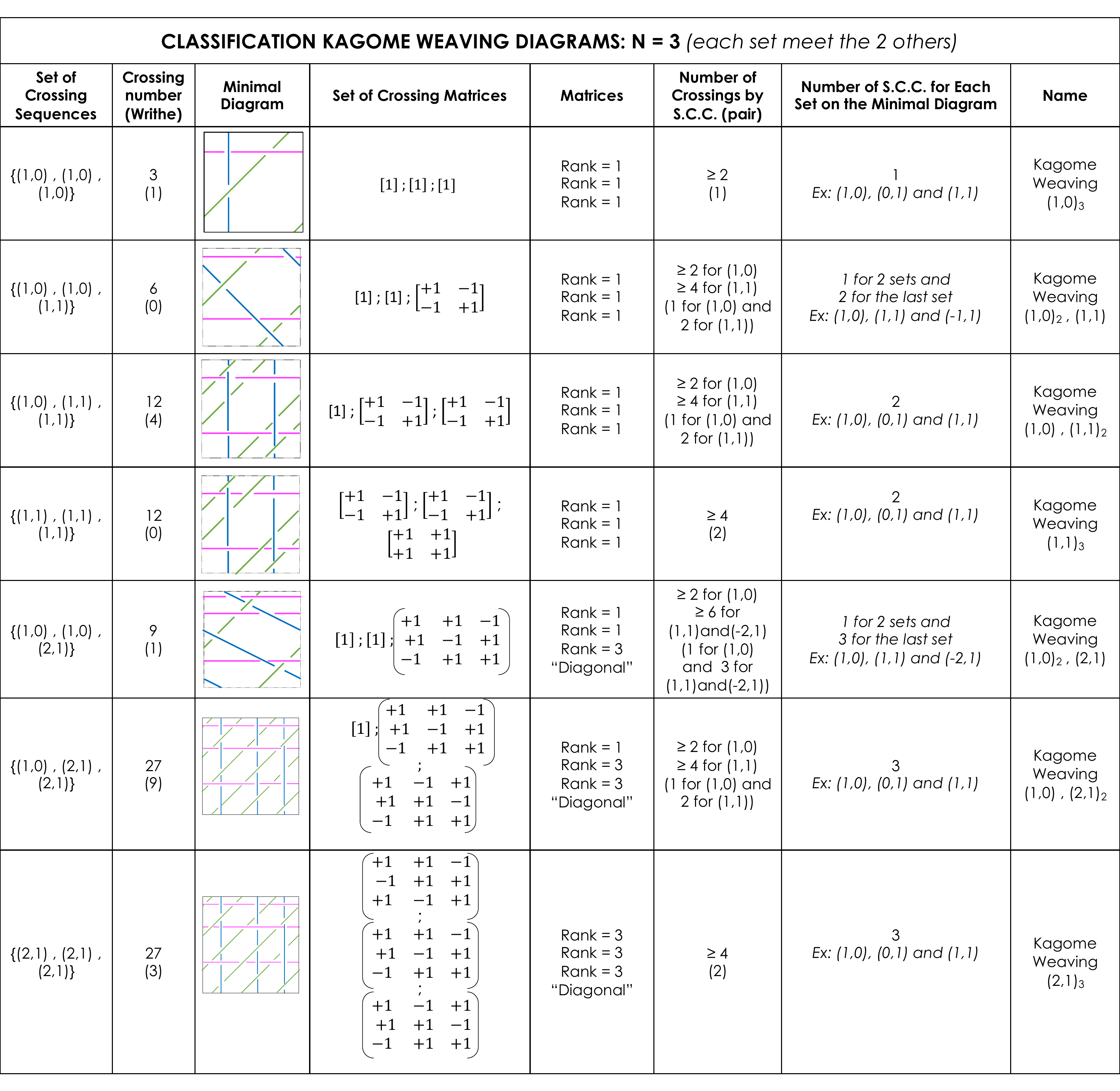}
      \caption{\label{fig12} Table Classification Kagome Weavings}
\end{figure}

\ ~

\ ~

\section{Acknoledgement}\label{sec:4}
We would like to thank A. Cheritat (IMT Toulouse), T. Kechadi (UCD Berlin), as well as M. Evans and her group (U. Potsdam and T.U. Berlin) for their precious comments and advice during this study. 
This work is supported by a Research Fellowship from JST CREST Grant Number JPMJCR17J4.\\




\end{document}